\documentclass[leqno,12pt]{amsart} %leqno the option to put formula numbers on the left side
\usepackage{amsmath, amsthm, amssymb}
\numberwithin{equation}{section}
\theoremstyle{plain} %text of this environment is typesetted in italics
\newtheorem{theorem}{\indent\sc Theorem}[section]
\newtheorem{lemma}[theorem]{\indent\sc Lemma}
\newtheorem{corollary}[theorem]{\indent\sc Corollary}
\newtheorem{proposition}[theorem]{\indent\sc Proposition}

\newtheorem{fact}[theorem]{\indent\sc Fact}
\theoremstyle{definition} %text of this environment is typesetted in roman letters
\newtheorem{definition}[theorem]{\indent\sc Definition}

\newtheorem{example}[theorem]{\indent\sc Example}

\def\C{{\mathbf{C}}}%   \C == \mathbf{C}
\def\R{{\mathbf{R}}}%   \R == \mathbf{R}
\def\H{{\mathbf{H}}}%   \H == \mathbf{H}
\def\N{{\mathbf{N}}}%   \N == \mathbf{N}
\def\L{{\mathbf{L}}}%  \L == \mathbf{L}
\def\Pi{{\mathbf{P}}}%  \Pi == \mathbf{P}
\def\Si{{\mathbf{S}}}%   \Si == \mathbf{S}
\def\E{{\mathcal{E}}}%   \E == \mathcal{E}

\def\tr#1{\mathord{\mathopen{{\vphantom{#1}}^t}#1}} 
\def\rank{\mathop{\mathrm{rank}}\nolimits}

\newcommand{\trace}{\mathop{\mathrm{trace}}\nolimits}

\begin{document}

\title[Flat fronts in hyperbolic three-space]{A ramification theorem for the ratio of canonical forms of flat surfaces in hyperbolic three-space} %title of paper and the running head option

\author[Y. Kawakami]{Yu Kawakami} %first author's name and the running head option

%\author[S. Author]{Second Author} %second author's name and the running head option

%\dedicatory{Dedicated to Professor Xxx Yyy on his sixtieth birthday}

%%%%%%%%%%%%%%% footnote %%%%%%%%%%%%%%%%
%\subjclass[2000]{ %2000 MSC numbers
%Primary 00; Secondary 00.
%}
%In case \subjclass[2000] command is not effective
%(or the version of amsart.cls is old), write as follows instead:
\renewcommand{\thefootnote}{\fnsymbol{footnote}}
\footnote[0]{2010\textit{ Mathematics Subject Classification}.
Primary 30D35, 53A35; Secondary 53C42.}
\date{5 January, 2012}
\thanks{ %acknowledgment of support etc. if any
The author is partially supported by the Grant-in-Aid for Young Scientists (B) No.~24740044, 
Japan Society for the Promotion of Science. 
}
%%%%%%%%%%%% Authors addresses %%%%%%%%%%%%%
\address{% First Author
Graduate School of Science and Engineering, \endgraf
Yamaguchi university, \endgraf
Yamaguchi, 753-8512, Japan
}
\email{y-kwkami@yamaguchi-u.ac.jp}

%\address{% Second Author
%Mathematical Institute \endgraf
%Tohoku University \endgraf
%Sendai 980-8578 \endgraf
%Japan
%}
%\email{author2@math.tohoku.ac.jp}

%\address{% Third Author
%Mathematical Institute \endgraf
%Tohoku University \endgraf
%Sendai 980-8578 \endgraf
%Japan
%}
%\email{author3@math.tohoku.ac.jp}

%\address{% Fourth Author
%Mathematical Institute \endgraf
%Tohoku University \endgraf
%Sendai 980-8578 \endgraf
%Japan
%}
%\email{author4@math.tohoku.ac.jp}
%%%%%%%%%%%%%%%%%%%%%%%%%%%%%%%%%%%%%%%%%

\maketitle

\begin{abstract}
We provide an effective ramification theorem for the ratio of canonical forms of a weakly complete flat front 
in the hyperbolic three-space. Moreover we give the two applications of this theorem, 
the first one is to show an analogue of the Ahlfors islands theorem for it 
and the second one is to give a simple proof of the classification of complete nonsingular flat surfaces 
in the hyperbolic three-space. 
\end{abstract}

\section*{Introduction} %delete * to number this section

It is well-known that any complete nonsingular flat surface in the hyperbolic 3-space ${\H}^{3}$ must be a horosphere or a hyperbolic cylinder, 
that is, a surface equidistance from a geodesic (\cite{Sa}, \cite{VV}). 
However if we consider flat fronts (namely, projections of Legendrian immersions) and define the notion of weakly completeness, 
there exist many examples and interesting global properties (for example, see \cite{KRUY1}, \cite{KRUY2} and \cite{MUY}). 

The ratio $\rho$ of canonical forms plays important roles in investigating the global properties of weakly complete flat fronts in ${\H}^{3}$. 
Indeed, Kokubu, Rossman, Saji, Umehara and Yamada \cite{KRSUY} showed that a point $p$ 
is a singular point of a flat front in ${\H}^{3}$ if and only if $|\rho (p)|=1$. Moreover the author and Nakajo \cite{KN} obtained the best
possible upper bound for the number of exceptional values of $\rho$ of a weakly complete flat front in ${\H}^{3}$. 

The purpose of the present paper is to study the value-distribution-theoretic properties of the ratio of canonical forms of weakly complete 
flat fronts in ${\H}^{3}$. The paper is organized as follows: In Section 1, we recall some definitions and fundamental properties 
of flat fronts in ${\H}^{3}$, which are used throughout this paper. In Section 2, we provide a ramification theorem for the ratio of canonical forms 
of a weakly complete flat front in ${\H}^{3}$ (Theorem \ref{main-thm}). The theorem is effective in the sense that it is sharp 
(see Corollary \ref{cor4-2-1} and the comment below) and has some applications. We note that it corresponds to the defect relation 
in Nevanlinna theory (\cite{Ko}, \cite{Ne}, \cite{NO} and \cite{Ru}). In Section 3, we give the two applications 
of this theorem. The first one is to show an analogue of a special case of the Ahlfors Islands Theorem \cite[Theorem B.2]{Ber} for the ratio of canonical 
forms of a weakly complete flat front in ${\H}^{3}$ (Corollary \ref{cor-4-1-2}). We remark that Klotz and Sario \cite{KS} investigated 
the number of islands for the Gauss map of minimal surfaces in the Euclidean 3-space ${\R}^{3}$. 
The second one is to give a simple proof of the classification of complete nonsingular flat surfaces in ${\H}^{3}$ 
(Corollary \ref{cor4-2-2}). 

Finally, the author would like to particular thank to Masatoshi Kokubu, Junjiro Noguchi, Yusuke Okuyama, Wayne Rossman, Masaaki Umehara, 
Kotaro Yamada and the referee for their useful advice and comments. 

%Acknowledgments may be included at the end of the Introduction.

%%%%%%%%%%%%%

\section{Preliminaries}
We briefly summarize here definitions and basic facts on flat fronts in ${\H}^{3}$ which we shall need. For more details, 
we refer the reader to \cite{GMM}, \cite{KaH}, \cite{KRSUY}, \cite{KRUY1}, \cite{KRUY2}, \cite{KUY1}, \cite{KUY2} and \cite{SUY}. 

%%% The hyperbolic 3-space %%%
Let ${\L}^{4}$ be the Lorentz-Minkowski 4-space with inner product of signature $(-,+,+,+)$. 
Then the hyperbolic 3-space is given by
\begin{equation}\label{hyperbolic-space}
{\H}^{3}=\{(x_{0}, x_{1}, x_{2}, x_{3})\in {\L}^{4} \, | \, -(x_{0})^{2}+(x_{1})^{2}+(x_{2})^{2}+(x_{3})^{2}=-1, x_{0}>0 \}
\end{equation}
with the induced metric from ${\L}^{4}$, which is a simply connected Riemannian 3-manifold with constant 
sectional curvature $-1$. 
Identifying ${\L}^{4}$ with the set of $2\times 2$ Hermitian matrices 
Herm($2$)$=\{X^{\ast}=X\}$ $(X^{\ast}:=\tr{\overline{X}}\,)$ by
\begin{equation}\label{Hermite}
(x_{0}, x_{1}, x_{2}, x_{3}) \longleftrightarrow \left(
\begin{array}{cc}
x_{0}+x_{3} & x_{1}+ix_{2} \\
x_{1}-ix_{2} & x_{0}-x_{3}
\end{array}
\right)
\end{equation}
where $i=\sqrt{-1}$, we can write
\begin{eqnarray}\label{hyperbolicspace}
\H^{3}&=& \{X\in \text{Herm(2)}\,;\, \det{X}=1, \trace{X}>0\} \\
      &=& \{aa^{\ast}\,;\, a\in SL(2,\C)\} \nonumber
\end{eqnarray}
with the metric 
\[
\langle X, Y \rangle = -\frac{1}{2}\trace{(X\widetilde{Y})}, \quad  \langle X, X \rangle =-\det(X)\, ,
\]
where $\widetilde{Y}$ is the cofactor matrix of $Y$. The complex Lie group $ PSL(2,\C):= SL(2,\C)/\{\pm \text{id} \}$ 
acts isometrically on 
$\H^{3}$ by 
\begin{equation}\label{action}
\H^{3} \ni X \longmapsto aXa^{\ast}\, , 
\end{equation}
where $a\in PSL(2,\C)$. 

%%% The definition of fronts %%% 
Let $\Sigma$ be an oriented 2-manifold. A smooth map $f\colon \Sigma\to {\H}^{3}$ is called a {\it front} 
if there exists a Legendrian immersion 
\[
L_{f}\colon \Sigma \to T_{1}^{\ast}{\H}^{3}
\]
into the unit cotangent bundle of ${\H}^{3}$ whose projection is $f$. 
Identifying $T_{1}^{\ast}{\H}^{3}$ with the unit tangent bundle $T_{1}{\H}^{3}$, 
we can write $L_{f}=(f, n)$, where $n (p)$ is a unit vector in $T_{f(p)}{\H}^{3}$ such that 
$\langle df(p), n (p) \rangle= 0$ for each $p\in M$. We call $n$ a {\it unit normal vector field} of the front $f$. 
A point $p\in\Sigma$ where $\rank{(df)}_{p}<2$ is called a {\it singularity} or {\it singular point}. 
A point which is not singular is called {\it regular point}, where the first fundamental form is positive definite.

%%% The definition of flat fronts %%%
The {\it parallel front} $f_{t}$ of a front $f$ at distance $t$ is given by $f_{t}(p)=\text{Exp}_{f(p)}(tn (p))$, 
where ``$\text{Exp}$'' denotes the exponential map of ${\H}^{3}$. 
In the model for ${\H}^{3}$ as in \eqref{hyperbolic-space}, we can write 
\begin{equation}\label{parallel}
f_{t}=(\cosh{t})f+(\sinh{t})n, \quad {n}_{t}=(\cosh{t})n +(\sinh{t})f\,,
\end{equation}
where ${n}_{t}$ is the unit normal vector field of $f_{t}$. 

Based on the fact that any parallel surface of a flat surface is also flat at regular points, 
we define flat fronts as follows: A front $f\colon \Sigma\to {\H}^{3}$ is said to be {\it flat} 
if, for each $p\in M$, there exists a real number $t\in \R$ such that the parallel front $f_{t}$ is a 
flat immersion at $p$. By definition, $\{f_{t}\}$ forms a family of flat fronts. 
We note that an equivalent definition of flat fronts is that the Gaussian curvature 
of $f$ vanishes at all regular points. However, there exists a case where this definition is not suitable. 
For details, see \cite[Remark 2.2]{KUY2}. 

We assume that $f$ is flat. Then there exists a (unique) complex structure on $\Sigma$ and 
a holomorphic Legendrian immersion 
\begin{equation}\label{Legen-lift}
{\E}_{f}\colon \widetilde{\Sigma}\to SL(2,\C)
\end{equation}
such that $f$ and $L_{f}$ are projections of ${\E}_{f}$, where $\widetilde{\Sigma}$ is 
the universal covering surface of $\Sigma$. 
Here ${\E}_{f}$ being a holomorphic Legendrian map means that ${\E}^{-1}_{f}d{\E}_{f}$ is off-diagonal (see \cite{GMM}, \cite{KUY1}, 
\cite{KUY2}). 
We call ${\E}_{f}$ the {\it holomorphic Legendrian lift} of $f$. 
The map $f$ and its unit normal vector field $n$ are 
\begin{equation}\label{Legen-map-vec}
f={\E}_{f}{\E}^{\ast}_{f}, \quad n = {\E}_{f}e_{3}{\E}^{\ast}_{f}, \quad e_{3}=\left(
\begin{array}{cc}
1 & 0 \\
0 & -1
\end{array}
\right)\,.
\end{equation}

If we set 
\begin{equation}\label{Legen-form}
{\E}^{-1}_{f}d{\E}_{f}=\left(
\begin{array}{cc}
0      & \theta \\
\omega & 0
\end{array}
\right)\, ,
\end{equation}
the first and second fundamental forms $ds^{2}=\langle df, df \rangle$ and 
$dh^{2}=-\langle df, dn \rangle$ are given by
\begin{eqnarray}\label{Legen-form2}
ds^{2}&=&|\omega+\bar{\theta}|^{2}=Q+\bar{Q}+(|\omega|^{2}+|\theta|^{2}), \quad Q=\omega\theta \\
dh^{2}&=&|\theta|^{2}-|\omega|^{2} \nonumber 
\end{eqnarray}
for holomorphic 1-forms $\omega$ and $\theta$ defined on $\widetilde{\Sigma}$, with $|\omega|^{2}$ and $|\theta|^{2}$ 
well-defined on $\Sigma$ itself. 
We call $\omega$ and $\theta$ the {\it canonical forms} of $f$. The holomorphic 2-differential $Q$ appearing 
in the $(2,0)$-part of $ds^{2}$ is defined on $\Sigma$, and is called the {\it Hopf differential} of $f$. 
By definition, the umbilic points of $f$ coincide with the zeros of $Q$. Defining a meromorphic function on $\widetilde{\Sigma}$ 
by the ratio of canonical forms 
\begin{equation}\label{sign-rho}
\rho=\dfrac{\theta}{\omega}\,,
\end{equation}
then $|\rho|\colon \Sigma\to [0, +\infty]$ is well-defined on $\Sigma$, and $p\in \Sigma$ is a singular point 
if and only if $|\rho(p)|=1$ (\cite{KRSUY}). 

Note that the $(1, 1)$-part of the first fundamental form 
\begin{equation}\label{eq-Sasakian}
ds^{2}_{1,1}=|\omega|^{2}+|\theta|^{2}=(1+|\rho|^{2})|\omega|^{2}
\end{equation}
is positive definite on $\Sigma$ because it is the pull-back of the canonical Hermitian metric of $SL(2, \C)$. 
Moreover $2ds^{2}_{1,1}$ coincides with the pull-back of the Sasakian metric on $T^{\ast}_{1}{\H}^{3}$ 
by the Legendrian lift $L_{f}$ of $f$ (which is the sum of the first and third fundamental forms in this case, 
see \cite[Section 2]{KUY2} for details). The complex structure on $\Sigma$ is compatible with the conformal 
metric $ds^{2}_{1,1}$. Note that any flat front is orientable (\cite[Theorem B]{KRUY1}). 
In the present paper, for each flat front $f\colon \Sigma\to {\H}^{3}$, 
we always regard $\Sigma$ as a Riemann surface with this complex structure. A flat front $f\colon \Sigma \to {\H}^{3}$ is said to be {\it weakly complete} if 
the metric $ds^{2}_{1,1}$ as in (\ref{eq-Sasakian}) is complete. We note that 
the universal cover of a weakly complete flat front is also weakly complete. 

Finally, we give examples which play important roles in the following sections. 

%%% Examples %%%
\begin{example}[{\cite[Example 4.1]{KUY2}}, flat fronts of revolution] Let $\overline{\Sigma}=\C\cup \{\infty\}$ and set 
$$
\omega =-\dfrac{1}{c^{2}}z^{-2/(1-\alpha)}dz, \quad \theta = \dfrac{c^{2}\alpha}{(1-\alpha)^{2}}z^{2\alpha /(1-\alpha)} dz, 
$$
for some constants $\alpha\in \R\backslash \{1\}$ and $c\in \R$. We define $\Sigma$ by $\Sigma =\overline{\Sigma}\backslash \{0\}$ for 
the case where $\alpha =0$ and $\Sigma =\overline{\Sigma}\backslash \{0, \infty\}$ for the case where $\alpha\not = 0$, respectively.  
Then we can construct a flat front $f\colon \Sigma\to {\H}^{3}$ whose canonical forms are $\omega$ and $\theta$. 
Indeed, these data give a Legendrian immersion 
$$
{\E}=\left(
\begin{array}{cc}
\dfrac{z^{-{\alpha}/(1-{\alpha})}}{c} & \dfrac{c{\alpha}z^{1/(1-\alpha)}}{1-\alpha} \smallskip \\
\dfrac{z^{-1/(1-{\alpha})}}{c} & \dfrac{cz^{{\alpha}/(1-\alpha)}}{1-\alpha}
\end{array}
\right) , 
$$
and the corresponding flat front $f={\E}{\E}^{\ast}$ is well-defined on $\Sigma$. 
Moreover $f$ is weakly complete because, for each end $p\in \overline{\Sigma}\backslash \Sigma$ of $f$, it holds that 
$$
\text{ord}_{p}ds^{2}_{1,1}=\min\{\text{ord}_{p}|\omega|^{2}, \text{ord}_{p}|\theta|^{2} \}\leq -1. 
$$
The ratio of canonical forms of $f$ is given by 
$$
\rho =\dfrac{\theta}{\omega}=-\dfrac{c^{4}\alpha}{1-\alpha}z^{2(1+\alpha)/(1-\alpha)}. 
$$
Thus if $\alpha=0$ or $-1$, then $\rho$ is constant. We note that $f$ is a horosphere if $\alpha =0$ or a hyperbolic cylinder if $\alpha =-1$. 
\end{example}

Moreover we can obtain weakly complete flat fronts in ${\H}^{3}$ of Voss type (\cite[Theorem 8.3]{Os}, \cite{Vo}). 

\begin{proposition}\label{prop-Voss}
Let $E$ be an arbitrary $q$ points on the Riemann sphere, where $q\leq 3$. Then there exists a weakly complete flat front in ${\H}^{3}$ 
whose image of the ratio of canonical forms omits precisely the set $E$. 
\end{proposition}
\begin{proof}
We set $E=\{{\alpha}_{1},\ldots, {\alpha}_{q-1}, {\alpha}_{q}\}\subset \C\cup\{\infty\}$ and $\Sigma:=\C\cup\{\infty\} \backslash E$. Then we may assume without 
loss of generality that ${\alpha}_{q}=\infty$. We take a holomorphic universal covering map $\xi\colon \widetilde{\Sigma}\to \Sigma$, 
where $\widetilde{\Sigma}$ is either the complex plane $\C$ or the unit disk. If we set 
$$
\omega= \dfrac{d\xi}{\prod_{i=1}^{q-1}(\xi -{\alpha}_{i})}, \quad \rho =\xi 
$$
and use the representation (\ref{Legen-lift}), (\ref{Legen-map-vec}) and (\ref{Legen-form}) on $\widetilde{\Sigma}$, we obtain a flat front 
in ${\H}^{3}$ whose the ratio of canonical forms omits precisely the points of $E$. Moreover it is weakly complete. Indeed, a divergent curve 
$\Gamma$ in $\widetilde{\Sigma}$ must tend to one of the point ${\alpha}_{i}$ $(1\leq i\leq q)$, and we have 
$$
\int_{\Gamma}ds_{1, 1}=\int_{\Gamma}\sqrt{1+|\rho|^{2}}|\omega|=\int_{\Gamma}\dfrac{\sqrt{1+|\xi|^{2}}}{\prod_{i=1}^{q-1}|\xi -{\alpha}_{i}|}|d\xi|=\infty\,, 
$$ 
when $q\leq 3$. 
\end{proof}

\section{Main theorem}

In this section, we give an effective ramification theorem for the ratio of canonical forms of a weakly complete flat front 
in ${\H}^{3}$. We first recall the case where the ratio is constant.

\begin{fact}{\cite[Propotion 4.4]{KN}}\label{main-fact1}
Let $f\colon\Sigma\to {\H}^{3}$ be a weakly complete flat front. If the ratio of canonical forms of $f$ defined by (\ref{sign-rho}) is 
constant, then $f$ is congruent to a horosphere or a hyperbolic cylinder. 
\end{fact}

The following is the main result of the present paper. 

\begin{theorem}\label{main-thm}
Let $f\colon \Sigma\to {\H}^{3}$ be a weakly complete flat front. Let $q\in \N$, ${\alpha}_{1}, \ldots, {\alpha}_{q}\in \C\cup\{\infty\}$ be distinct 
and $m_{1}, \cdots, m_{q}\in \N\cup \{\infty \}$. Suppose that 
\begin{equation}\label{eq2-1}
\gamma =\displaystyle \sum_{j=1}^{q} \biggl(1-\dfrac{1}{m_{j}} \biggr)> 3. 
\end{equation}
If the ratio of canonical forms $\rho\colon \widetilde{\Sigma}\to \C\cup\{\infty\}$ of $f$ satisfies the property that all ${\alpha}_{j}$-points of $\rho$ 
have multiplicity at least $m_{j}$, then $f$ must be congruent to a horosphere or a hyperbolic cylinder. 
\end{theorem}

Before proceeding to the proof of Theorem \ref{main-thm}, we recall two function-theoretical lemmas. 
For two distinct values $\alpha$, $\beta\in \C\cup\{\infty\}$, we set 
$$
|\alpha, \beta|:=\dfrac{|\alpha -\beta|}{\sqrt{1+|\alpha|^{2}}\sqrt{1+|\beta|^{2}}}
$$ 
if $\alpha\not= \infty$ and $\beta\not= \infty$, and $|\alpha, \infty|=|\infty, \alpha|:=1/\sqrt{1+|\alpha|^{2}}$. 
Note that, if we take $v_{1}$, $v_{2}\in {\Si}^{2}$ with $\alpha =\varpi (v_{1})$ and $\beta =\varpi (v_{2})$, 
we have that $|\alpha, \beta|$ is a half of the chordal distance between $v_{1}$ and $v_{2}$, where $\varpi$ denotes 
the stereographic projection of the 2-sphere ${\Si}^{2}$ onto $\C\cup\{\infty\}$.  

\begin{lemma}[{\cite[Corollary 1.4.15]{Fu2}}]\label{main-lem1}
Let $\rho$ be a nonconstant meromorphic function on ${\triangle}_{R}=\{z\in \C; |z|< R\}$ $(0<R\leq \infty)$. 
Let $q\in \N$, ${\alpha}_{1}, \ldots, {\alpha}_{q}\in \C\cup\{\infty\}$ be distinct 
and $m_{1}, \cdots, m_{q}\in \N\cup \{\infty \}$. Suppose that 
$$
\gamma =\displaystyle \sum_{j=1}^{q} \biggl(1-\dfrac{1}{m_{j}} \biggr)> 2. 
$$
If $\rho$ satisfies the property that all ${\alpha}_{j}$-points of $\rho$ have multiplicity at least $m_{j}$, then, for 
arbitrary constants $\eta\geq 0$ and $\delta >0$ with $\gamma -2>\gamma(\eta +\delta)$, it holds that 
\begin{equation}\label{eq2-3}
\dfrac{|{\rho}'|}{1+|\rho|^{2}}\dfrac{1}{({\prod}_{j=1}^{q}|\rho, {\alpha}_{j}|^{1-1/m_{j}})^{1-\eta-\delta}}\leq C\dfrac{R}{R^{2}-|z|^{2}}, 
\end{equation}
where $C$ is some constant depending only on $\gamma$, $\eta$, $\delta$, and $L:={\min}_{i<j}|{\alpha}_{i}, {\alpha}_{j}|$. 
\end{lemma}

\begin{lemma}[{\cite[Corollary 1.6.7]{Fu2}}]\label{main-lem2}
Let $d{\sigma}^{2}$ be a conformal flat metric on an open Riemann surface $\Sigma$. 
Then, for each point $p\in \Sigma$, there exists a local diffeomorphism $\Phi$ of a 
disk ${\Delta}_{R}=\{z\in \C\, ;\, |z|<{R}\}$ $(0<{R}\leq +\infty)$ onto an open neighborhood 
of $p$ with $\Phi (0)=p$ such that $\Phi$ is a local isometry, that is, the pull-back ${\Phi}^{\ast}(d{\sigma}^{2})$ 
is equal to the standard Euclidean metric $ds_{Euc}^{2}$ on  ${\Delta}_{R}$ and, for a point $a_{0}$ with $|a_{0}|=1$, 
the $\Phi$-image ${\Gamma}_{a_{0}}$ of the curve $L_{a_{0}}=\{w:=a_{0}s\, ;\, 0<s<R \}$ is divergent in $\Sigma$. 
\end{lemma}

\begin{proof}[{\it Proof of Theorem \ref{main-thm}}]
This is proved by contradiction. Suppose that $\rho$ is nonconstant. For our purpose, 
we may assume ${\alpha}_{q}=\infty$ after a suitable M\"obius transformation and that $\widetilde{\Sigma}$ 
is biholomorphic to the unit disk because Theorem \ref{main-thm} is obvious in the case where $\widetilde{\Sigma}=\C$ by Nevanlinna theory 
\cite[Section 3 in Chapter X]{Ne}. 
We choose some $\delta$ such that $\gamma -3 > 2{\gamma}^{2}\delta > 2\gamma\delta >0$ and set 
\begin{equation}\label{eq2-4}
\eta :=\dfrac{\gamma -3-2\gamma\delta}{\gamma}, \quad \lambda=\dfrac{1}{1+\gamma\delta}\,. 
\end{equation}
Then if we choose a sufficiently small positive number $\delta$ depending only on $\gamma$, for the constant ${\varepsilon}_{0}=(\gamma -3)/2\gamma$ 
we have 
\begin{equation}\label{eq2-5}
0<\lambda< 1, \quad \dfrac{{\varepsilon}_{0}\lambda}{1-\lambda}\biggl{(}=\dfrac{\gamma -3}{2{\gamma}^{2}\delta}\biggr{)}>1\,.
\end{equation}
Now we define a new metric 
\begin{equation}\label{eq2-6}
d{\sigma}^{2}=\dfrac{1}{|{\rho'}_{z}|^{\lambda}}\biggl{|}\dfrac{\omega}{d\rho}\biggr{|}^{2/(1-\lambda)}\biggl{(}\prod_{j=1}^{q-1}\biggl{(}\dfrac{|\rho -{\alpha}_{j}|}
{\sqrt{1+|{\alpha}_{j}|^{2}}}\biggr{)}^{{\eta}_{j}(1-\eta -\delta)} \biggr{)}^{2\lambda/(1-\lambda)}|dz|^{2}
\end{equation}
on the set $\widetilde{\Sigma}'=\{z\in\widetilde{\Sigma} ;\,{\rho}_{z}'(z)\not=0\: \text{and}\: \rho(z)=a_{j}\; \text{for all}\; j \}$ where 
$\omega =h_{z}dz$, ${\rho}_{z}'=d\rho /dz$ and ${\eta}_{j}=1-1/m_{j}$. Take a point $p\in\widetilde{\Sigma}'$. Since the metric $d{\sigma}^{2}$ is 
flat on $\widetilde{\Sigma}'$, by Lemma \ref{main-lem2}, there exists a local isometry $\Phi$ satisfying $\Phi (0)=p$ from a disk 
${\triangle}_{R}=\{z\in\C ;\,|z|<R\}$ $(0<R\leq +\infty)$ with the standard metric $ds^{2}_{Euc}$ onto an open neighborhood of $p$ in $\widetilde{\Sigma}'$ 
with the metric $d{\sigma}^{2}$ such that, for a point $a_{0}$ with $|a_{0}|=1$, the $\Phi$-image ${\Gamma}_{a_{0}}$ of the curve 
$L_{a_{0}}=\{w:=a_{0}s\, ;\, 0<s<R \}$ is divergent in $\widetilde{\Sigma}'$. For brevity, we denote the function $\rho\circ \Phi$ on ${\triangle}_{R}$ 
by $\rho$ in the followings. By Lemma \ref{main-lem1}, we get 
\begin{equation}\label{eq2-7}
R\leq C\dfrac{1+|\rho (0)|^{2}}{|{\rho}_{z}'(0)|}\prod_{j=1}^{q}|\rho (0), {\alpha}_{j}|^{{\eta}_{j}(1-\eta -\delta)} < +\infty\,.
\end{equation}
Hence 
$$
L_{d\sigma}({\Gamma}_{a_{0}})=\int_{{\Gamma}_{a_{0}}}d\sigma =R <+\infty, 
$$
where $L_{d\sigma}({\Gamma}_{a_{0}})$ denotes the length of ${\Gamma}_{a_{0}}$ with respect to the metric $d{\sigma}^{2}$. 

Now we prove that ${\Gamma}_{a_{0}}$ is divergent in $\widetilde{\Sigma}$. If not, 
then ${\Gamma}_{a_{0}}$ must tend to a point $p_{0}\in \widetilde{\Sigma}\backslash \widetilde{\Sigma}'$ where ${\rho}_{z}'(p_{0})=0$ or 
$\rho (p_{0})={\alpha}_{j}$ for some $j$ because ${\Gamma}_{a_{0}}$ is divergent in $\widetilde{\Sigma}'$ and $L_{d\sigma}({\Gamma}_{a_{0}})< +\infty$. 
Taking a local complex coordinate $\zeta$ in a neighborhood of $p_{0}$ with $\zeta (p_{0})=0$, we can write the metric $d{\sigma}^{2}$ as
$$
d{\sigma}^{2}=|\zeta|^{2k\lambda/(1-\lambda)}v|d\zeta|^{2}
$$
with some positive smooth function $v$ and some real number $k$. If $\rho -{\alpha}_{j}$ has a zero of order $m(\geq m_{j}\geq 2)$ at $p_{0}$ for 
some $j\leq q-1$, then ${\rho}'_{z}$ has a zero of order $m-1$ at $p_{0}$ and $h_{z}(p_{0})\not= 0$. Then we have 
\begin{eqnarray}
k &=& -(m-1)+m\biggl{(}1-\dfrac{1}{m_{j}} \biggr{)}(1-\eta -\delta) \nonumber \\
  &=& 1-\dfrac{m}{m_{j}}-\dfrac{m}{m_{j}}(m_{j}-1)(\eta +\delta) \nonumber \\
  &\leq & -(\eta +\delta)\leq -{{\varepsilon}_{0}}\,. \nonumber
\end{eqnarray}
For the case where $\rho$ has a pole of order $m(\geq m_{q})$, ${\rho}'_{z}$ has a pole of order $m+1$, $h_{z}$ has a zero of order $m$ at $p_{0}$ 
and each component $\rho -{\alpha}_{j}$ in the right side of (\ref{eq2-6}) has a pole of order $m$ at $p_{0}$. 
Using the identity ${\eta}_{1}+\cdots +{\eta}_{q-1}=\gamma -{\eta}_{q}$ and (\ref{eq2-5}), we get 
\begin{eqnarray}
k &=& \dfrac{m}{\lambda} +(m+1)-m(\gamma -{\eta}_{q})(1-\eta -\delta) \nonumber \\
  &=& m{\eta}_{q}(1-\eta -\delta)-(m-1)\leq -{\varepsilon}_{0}\,. \nonumber
\end{eqnarray}
Moreover, for the case where ${\rho}_{z}'(p_{0})=0$ and ${\rho}(p_{0})\not= {\alpha}_{j}$ for all $j$, we see $k\leq -1$. 
Thus, in any case, $k\lambda /(1-\lambda)\leq -1$ by (\ref{eq2-5}) and there exists a positive constant $C'$ such that 
$$
d\sigma \geq C'\dfrac{|d\zeta|}{|\zeta|}
$$
in a neighborhood of $p_{0}$. Hence we have 
$$
R=\int_{{\Gamma}_{a_{0}}}d\sigma \geq C'\int_{{\Gamma}_{a_{0}}}\dfrac{|d\zeta|}{|\zeta|}= +\infty, 
$$
which contradicts (\ref{eq2-7}). 

On the other hand, since ${\Phi}^{\ast}d{\sigma}^{2}=|dz|^{2}$, we obtain by (\ref{eq2-6})
\begin{equation}\label{eq2-8}
|\omega|=\biggl{(}|{\rho}_{z}'|\prod_{j=1}^{q-1}\biggl{(}\dfrac{\sqrt{1+|{\alpha}_{j}|^{2}}}{|\rho -{\alpha}_{j}|} \biggr{)}^{{\eta}_{j}(1-\eta -\delta)} \biggr{)}^{\lambda}|dz|\,.
\end{equation}
By Lemma \ref{main-lem1}, we have 
\begin{eqnarray}
{\Phi}^{\ast}ds_{1,1} &=& \sqrt{1+|{\rho}|^{2}}|\omega| \nonumber \\
                      &=& \biggl{(}|{\rho}_{z}|(1+|\rho|^{2})^{1/2\lambda}\prod_{j=1}^{q-1}\biggl{(}\dfrac{\sqrt{1+|{\alpha}_{j}|^{2}}}{|\rho -{\alpha}_{j}|} \biggr{)}^{{\eta}_{j}(1-\eta -\delta)} \biggr{)}^{\lambda} |dz| \nonumber \\
                      &=& \biggl{(}\dfrac{|{\rho}_{z}'|}{1+|\rho|^{2}}\dfrac{1}{{\prod}_{j=1}^{q}|\rho , {\alpha}_{j}|^{{\eta}_{j}(1-\eta -\delta)}} \biggr{)}^{\lambda} |dz| \nonumber \\
                      &\leq & C^{\lambda}\biggl{(}\dfrac{R}{R^{2}-|z|^{2}} \biggr{)}^{\lambda}|dz|\,. \nonumber
\end{eqnarray}
Thus if we denote the distance $d(p)$ from a point $p\in \widetilde{\Sigma}$ to the boundary of $\widetilde{\Sigma}$ as the greatest lower bound of 
the lengths with respect to the metric $ds_{1,1}^{2}$ of all divergent paths in $\widetilde{\Sigma}$, then we have 
$$
d(p)\leq \int_{{\Gamma}_{a_{0}}}ds_{1,1}=\int_{L_{a_{0}}}{\Phi}^{\ast}ds_{1,1}\leq C^{\lambda}\int_{L_{a_{0}}}\biggl{(}\dfrac{R}{R^{2}-|z|^{2}} \biggr{)}^{\lambda}|dz|\leq C^{\lambda}\dfrac{R^{1-\lambda}}{1-\lambda}< +\infty
$$
because $0<\lambda <1$. However it contradicts the assumption that $ds_{1,1}^{2}$ is complete. 
\end{proof}

\section{Applications} 
This section is devoted to prove two applications of the main theorem.
\subsection{The Ahlfors Islands Theorem}
We recall the definition of an island of a meromorphic function on a Riemann surface. 

\begin{definition}
Let $\Sigma$ be a Riemann surface and $h\colon \Sigma\to \C\cup\{\infty\}$ a meromorphic function. 
Let $V\subset \C\cup\{\infty\}$ be a Jordan domain. A simply-connected component $U$ of $h^{-1}(V)$ with $\overline{U}\subset \Sigma$ is 
called an {\it island} of $h$ over $V$. Note that $h|_{U}\colon U\to V$ is a proper map. The degree of this map is called the {\it multiplicity} of 
the island $U$. An island of multiplicity one is called a {\it simple island}. 
\end{definition} 

Since the ratio $\rho$ of canonical forms of a flat front $f\colon \Sigma\to {\H}^{3}$ is a meromorphic function on $\widetilde{\Sigma}$, 
we can consider an island of $\rho$. When all islands of $\rho$ are small disks, we get the following result by applying Theorem \ref{main-thm}. 

\begin{corollary}\label{cor-4-1-1}
Let $f\colon \Sigma\to {\H}^{3}$ be a weakly complete flat front. Let $q\in \N$, ${\alpha}_{1}, \ldots, {\alpha}_{q}\in \C$ be distinct, 
$D({\alpha}_{j}, \varepsilon):=\{z\in \C : |z-{\alpha}_{j}|<\varepsilon \}$ $(1\leq j\leq q)$ be pairwise disjoint and $m_{1}, \ldots, m_{q}\in \N$. 
Suppose that 
\begin{equation}
\displaystyle \sum_{j=1}^{q} \biggl(1-\dfrac{1}{m_{j}} \biggr)> 3. 
\end{equation}
Then there exists $\varepsilon >0$ such that if the ratio of canonical forms of $f$ has no island of 
multiplicity less than $m_{j}$ over $D({\alpha}_{j}, \varepsilon)$ for all $j\in\{1, \ldots, q\}$ then $f$ must be congruent to a horosphere or 
a hyperbolic cylinder. 
\end{corollary}
\begin{proof}
If such an $\varepsilon$ does not exist, for any $\varepsilon$, we can find a weakly complete flat front whose the ratio of canonical forms $\rho$ is
nonconstant and has no island of multiplicity less than $m_{j}$ over $D({\alpha}_{j}, \varepsilon)$. However this implies that all ${\alpha}_{j}$-points 
of $\rho$ have multiplicity at least $m_{j}$, contradicting Theorem \ref{main-thm}.
\end{proof}

The important special case of Corollary \ref{cor-4-1-1} is the case where $q=7$ and $m_{j}=2$ for all $j$. This corresponds to a weak version of 
the so-called Five Islands Theorem in the Ahlfors theory of covering surfaces (\cite{Ahl}, \cite[Chapter XIII]{Ne}). 

\begin{corollary}\label{cor-4-1-2}
Let $f\colon \Sigma\to {\H}^{3}$ be a weakly complete flat front. Let ${\alpha}_{1}, \ldots, {\alpha}_{7}\in \C$ be distinct and 
$D({\alpha}_{j}, \varepsilon):=\{z\in \C : |z-{\alpha}_{j}|<\varepsilon \}$ $(1\leq j\leq 7)$. 
Then there exists $\varepsilon >0$ such that if the ratio of canonical forms $\rho$ of $f$ has no simple island of over any of the 
small disks $D({\alpha}_{j}, \varepsilon)$ then $f$ must be congruent to a horosphere or a hyperbolic cylinder. 
\end{corollary}

\subsection{The classification of complete flat surfaces in ${\H}^{3}$} 
As another application of Theorem \ref{main-thm}, we obtain the best possible upper bound for the number of exceptional values of the ratio of 
canonical forms of a weakly complete flat front in ${\H}^{3}$. 

\begin{corollary}[{\cite[Theorem 4.5]{KN}}]\label{cor4-2-1}
Let $f$ be a weakly complete flat front in ${\H}^{3}$. If the ratio of canonical forms $\rho$ of $f$ omits more than three values, 
then $f$ must be congruent to a horosphere or a hyperbolic cylinder. 
\end{corollary}
\begin{proof}
In Theorem \ref{main-thm}, if $\rho$ does not take a value ${\alpha}_{j}$, then we may set $m_{j}=\infty$ in (\ref{eq2-1}), and 
if $\rho$ omits all values ${\alpha}_{j}$ $(1\leq j\leq q)$, (\ref{eq2-1}) means $q>3$, which is the case of this result. 
\end{proof}

The number ``three'' is sharp because there exist examples in Proposition \ref{prop-Voss}. As an application of this corollary, we give a simple proof 
of the classification of complete nonsingular flat surfaces in ${\H}^{3}$.  

\begin{corollary}[\cite{Sa}, \cite{VV}]\label{cor4-2-2}
Any complete flat surface in ${\H}^{3}$ must be congruent to a horosphere or a hyperbolic cylinder. 
\end{corollary}
\begin{proof}
Because a flat surface has no singularities, the complement of the image of $\rho$ contains at least the set $\{|\rho|=1\}\subset \C\cup\{\infty\}$. 
On the other hand, Kokubu, Umehara and Yamada \cite[Corollary 3.4]{KUY2} proved that a complete flat surface in ${\H}^{3}$ is also 
weakly complete. Therefore, by Corollary \ref{cor4-2-1}, it must be congruent to a horosphere or a hyperbolic cylinder.
\end{proof}

%%%%%%%%%%%% References %%%%%%%%%%%%%
%%
%<Author name> is written as Initial of Given Name, and Family Name.
%<Title> is written in roman letters.
%<Journal name> should be abbreviated according to
% the MR Serials Abbreviations List of Mathematical Reviews:
% (Abbreviations of Names of Serials; http://www.ams.org/mr-database)
%For <Pages>, use en-dash "--" between page numbers.
%%

\end{document}